\documentclass{compositio}
\usepackage{amsmath}
\usepackage{enumerate}
\usepackage{amsfonts}
\usepackage{enumerate}
\usepackage{amsthm}
\usepackage{latexsym}
\usepackage{amssymb}
\usepackage[all]{xy}
\usepackage{psfrag}
\usepackage{epsfig}
\usepackage{color}
\usepackage[all]{xy}
\usepackage{hyperref}

\usepackage{mathbbol}
\usepackage{bold-extra}
\usepackage{comment}
\numberwithin{equation}{section}
\begin{document}

% Enter full title and short title for running headers
\title[Poisson distribution corresponding to elliptic curves]{Poisson distribution of a prime counting function corresponding to elliptic curves}

% Enter the publication year and the ID number of the paper
\author{R. Balasubramanian}
\email{balu@imsc.res.in}
\address{Department of Mathematics, Institute of Mathematical Sciences, Chennai, India-600113}

\author{Sumit Giri}
\email{sumit.giri199@gmail.com}
\address{School of Mathematics, Tel Aviv University, P.O.B. 39040, Ramat Aviv, Tel Aviv 69978, Israel.}


\newtheorem{theoremm}{Theorem}
\newtheorem{theorem}{Theorem}
\newtheorem{lemma}{Lemma}
\newtheorem{rem}{Remark}
\newtheorem{note}{Note}
\newtheorem{prop}{Proposition}
\newtheorem{cor}{Corollary}
\newtheorem{defn}{Definition}
\newtheorem{conj}{Conjecture}

\renewcommand{\thetheoremm}{\Alph{theoremm}}
\renewcommand{\theenumi}{\Alph{enumi}}
\newcommand{\N}{\mathbb{N}}
\newcommand{\Z}{\mathbb{Z}}
\newcommand{\Q}{\mathbb{Q}}
\newcommand{\R}{\mathbb{R}}
\newcommand{\C}{\mathcal{C}}
\newcommand{\PP}{\mathcal{P}}
\newcommand{\m}{\Z /m \Z}
\newcommand{\F}{\mathbb{F}}
\newcommand{\f}{\hat{F}}
\newcommand{\g}{\hat{G}}
\newcommand{\D}{\Z /d \Z}
\newcommand{\n}{\Z /n \Z}
\newcommand{\A}{\mathcal{A}}
\newcommand{\h}{\mathcal{H}}
\newcommand{\B}{\mathcal{B}}
\newcommand{\OO}{\mathcal{O}}
\newcommand{\p}{\textfrak{p}}
\newcommand{\q}{\textfrak{q}}
\newcommand{\bb}{\textfrak{b}}
\newcommand{\aaa}{\textfrak{a}}
\newcommand{\pr}{\displaystyle \prod}
\newcommand{\s}{\displaystyle \sum}

\begin{abstract}
Let $E$ be an elliptic curve defined over rational field $\Q$ and $N$ be a positive integer. Now $M_E(N)$ denotes the number of primes $p$, such that the group $E_p(\F_p)$ is of order $N$. We show that $M_E(N)$ follows Poisson distribution when an average is taken over a large class of curves.
\end{abstract}

\maketitle

\section{Introduction}
Let $E$ be an elliptic curve defined over the field of rationals $\Q$. For a prime $p$ where $E$ has good reduction, we denote by $E_p$ the reduction of $E$ modulo $p$.
Let $\F_p$ be the finite field with $p$ elements and $E_p(\F_{p})$ be the group of $\F_p$ points over $E_p$. \par
We know $|E_p(\F_p)|=p+1-a_p(E)$ where $a_p(E)$ is the trace of the Frobenius morphism at $p$. By Hasse's theorem we know $|a_p(E)|<2\sqrt{p}$.
Also, it is well known that $E_p(\F_p)$ admits the structure of an abelian group of the form $\Z/m\Z\times \Z/mk\Z$, where $m$ divides $(p-1)$. We denote such groups by $G_{m,k}$. The question related to density of elliptic curve groups among all groups of the form $G_{m,k}$ has been addressed in \cite{Banks_Pappalardi_Shparlinski:2012}. Also, the question related to the primality of $E_p(\F_p)$  has been discussed in \cite{Balog_Cojocaru:2011}. \par

For a fixed positive integer $N$, we define the following prime counting function  
\begin{align}
 M_E(N):=\#\{p \text{ prime}:\, E\text{ has good reduction over $p$ and } |E_p(\F_p)|=N \}. \label{eq_1.1}
\end{align}
Here we note that the Hasse's theorem implies 
\begin{align}
 (\sqrt{p}-1)^2<&N<(\sqrt{p}+1)^2\nonumber
\intertext{or equivalently, }
N^-:=(\sqrt{N}-1)^2<&p<(\sqrt{N}+1)^2:=N^+.\label{eq_2}
\end{align}
This in turn implies that
\begin{align}
M_E(N)\ll \frac{\sqrt{N}}{\log (N+1)}.\label{eq_1}
\end{align} 
Using Chinese Reminder theorem, it is not difficult to construct a curve $E$ such that the upper bound in (\ref{eq_1}) is attained.

Also, it is not difficult to prove that \begin{align}
                                        \sum_{N\le x}M_E(N)=\pi(x)+O(\sqrt{x}).\label{eq_1.4}
                                       \end{align}

 Consequently, $M_E(N)$ is zero for most of the $N$'s. Under the assumption that $E_p(\F_p)$ is uniformly distributed over the range $[p^-,p^+]$, $M_E(N)$ is expected to be $\sim \frac{c}{\log N}$. See equation (4) in \cite{David_Smith:2013} for more details.\par
Now, for a pair of integers $(a,b)$, let $E_{a,b}$ be the elliptic curve defined by the Weierstrass equation
$$E_{a,b}:y^2=x^3+ax+b.$$
Also for $A,B>0$, we define the class of curves $\C(A,B)$  by \begin{align}
                                                               \C(A,B):=\{E_{a,b}:\, |a|\le A, |b|\le B, \Delta(E_{a,b})\neq 0\}.\label{eq_1.6}
                                                              \end{align}

 Now us recall Barban-Davenport-Halberstam conjecture 
\begin{conj}[\emph{BDH}]\label{conj_1} Let $\theta(x; q, a) = \underset{p\le x, {p\equiv a (\text{mod }q)}}{\sum} \log p$. Let $0 < \eta \le 1$ and $\beta > 0$ be real numbers. Suppose that $X$, $Y$, and $Q$ are positive real numbers satisfying $X^\eta \le Y \le X$
and $Y /(\log X)^\beta \le Q \le Y$. Then
$$\sum_{q\le Q }\sum_{\underset{(a,q)=1}{1\le a\le q}}|\theta(X + Y ; q, a)- \theta(X; q, a) -\frac{Y}{\phi(q)}|^2\ll_{\eta,\beta} Y Q \log X.$$
 \end{conj}
 
 Under the above hypothesis concerning short interval distribution of primes in arithmetic progressions, David and Smith\cite{David_Smith:2013,David_Smith:2014} proved that 
 \begin{theoremm}\label{thmm_1} Let Conjecture \ref{conj_1} be true for some $0<\eta<\frac{1}{2}$.
If $A, B\ge \sqrt{N}(\log N)^{1+\gamma}\log \log N$ and that $AB\ge N^{\frac{3}{2}}(\log N)^{2+\gamma}\log \log N$, then for any odd integer $N$, we have 
\begin{align}
 \frac{1}{\#\C(A,B)}\sum_{E\in \C(A,B)}M_E(N)&=K(N)\frac{N}{\phi(N)\log N}+O(\frac{1}{(\log N)^{1+\gamma}}),
\intertext{with }
K(N):=\prod_{p\nmid N}\left( 1-\frac{(\frac{N-1}{p})^2p+1}{(p-1)^2(p+1)}\right)&\prod_{{p\mid N}}\left(1-\frac{1}{p^{\nu_p(N)}(p-1)}\right), \label{eq_1.5} \end{align}
where $\nu_p$ denotes the usual $p$-adic valuation where $\nu_p(n)$ and $\left(\frac{n-1}{p}\right)$ is the Kronecker symbol.\par 
\end{theoremm}
In \cite{Chandee_David_Koukoulopoulos_Smith:2015}, Chandee, David, Koukoulopoulos and Smith extended this result over all $N$. \par
From [Theorem 1.7, \cite{Chandee_David_Koukoulopoulos_Smith:2015}], unconditionally, we also have 
\begin{align}
\frac{1}{\#\mathcal{C}(A,B)}\underset{E\in \mathcal{C}(A,B)}{\sum}M_E(N)\ll \frac{N}{\phi(N)\log N}=O\left(\frac{\log \log N}{\log N}\right)\label{eq_09} 
\end{align}
for large enough $A,B$.\\
Note that the above theorem is based on the assumption of Barban-Davenport-Halberstam conjecture for a particular range.
Martin, Pollack and Smith\cite{Martin_Pollack_Smith:2014} and authors\cite{Balasubramanian_Giri:2015} independently computed the mean value of $NK(N)/\phi(N)$ to show that Theorem \ref{thmm_1} is consistent with (\ref{eq_1.4}).\par 
In this paper we try to focus on the distribution of the function $M_E(N)$. In other words, if $N$ is a fixed integer and $E$ be any arbitrary chosen curve from a large class of curves, then what is the probability of the event $\{M_E(N)=\ell \}$ where $\ell$ is a positive integer.\par  
Under the assumption that primes are randomly distributed and reduction modulo two different primes are two independent events, one would expect the event $\{E\in \C: M_E(N)=\ell\}$ occurs with a probability $\sim \frac{1}{(\log N)^{\ell}}$. The main theorem of this paper is

\begin{theorem}\label{th_1}
 Let $\C(A,B)$ be as defined as in (\ref{eq_1.6}) and $N$ be a positive integer greater than $7$. If $L$ be a positive integer such that $A,B>N^{L/2}(\log N)^{1+\gamma}$ and $AB>N^{3L/2}(\log N)^{2+\gamma}$ for some $\gamma>0$, then for $1\le \ell\le L$ 
 \begin{align*}
  \frac{1}{\#\C(A,B)}\sum_{\underset{M_E(N)=\ell}{E\in \C(A,B)}}1=\frac{1}{\ell!}\left(\frac{1}{\#\C(A,B)}\sum_{E\in \C(A,B)}M_E(N)\right)^\ell\left(1+O\left(\frac{N}{\phi(N)\log N}\right)\right)+O\left(\frac{1}{N^{\frac{L-\ell}{2}}(\log N)^{\gamma}}\right),
 \end{align*}
where the \textquoteleft$O$\textquoteright constant in the last error term is independent of $\gamma$.
 \end{theorem}
Now we know that if $X_N\sim \text{Poisson}(\lambda_N)$, for $N=1,2,\cdots$, then the probability mass function of $X_N$ is 
$$f_{X_N}(\ell)=\frac{(\lambda_N)^\ell e^{-{\lambda_N}}}{\ell !} \quad \text{ for }\ell=0,1,2,\cdots .$$ 
If we take $\lambda_N=\frac{1}{\#\mathcal{C}(A,B)}\underset{E\in \mathcal{C}(A,B)}{\sum}M_E(N)$, then in view of (\ref{eq_09}), one can see that if $L$ is large, then on an average $M_E(N)$ follows a limiting Poisson distribution with mean $\frac{1}{\#\mathcal{C}(A,B)}\underset{E\in \mathcal{C}(A,B)}{\sum}M_E(N)$ as $N\rightarrow \infty$. The integer $L$ in \emph{Theorem \ref{th_1}} is introduced to ensure the finiteness of the class $\mathcal{C}(A,B)$.

One can immediately see  that if one also assumes Conjecture \ref{conj_1}, as in Theorem \ref{thmm_1}, then the right hand side is asymptotic to $\frac{1}{\ell!}\left(\frac{NK(N)}{\phi(N)\log N}\right)^\ell$. 

In \cite{Kowalski:2006}, Kowalski raised a question related to the behavior of sums of the type 

\begin{align*}
 \sum_{N\le x}M_E(N)^r\quad \text{and} \sum_{\underset{M_E(N)\ge 2}{N\le x}}M_E(N).
\end{align*}

To answer this question, we start with the quantity 

\begin{align}\frac{1}{\#\C}\sum_{E\in \C}\sum_{\underset{N\le x}{M_E(N)\ge \ell}}M_E(N)^r&\label{eq_12}\end{align}
for two non negative integers $r$ and $\ell$.\\
Before stating our result related to (\ref{eq_12}), we shall introduce a sequence of constants $\{C(m)\}_{m=\ell}^\infty$, where $C(m)$ corresponds to the $m-$th moment of the function $NK(N)/\phi(N)$ where $K(N)$ as defined in (\ref{eq_1.5}).
More precisely,

\begin{align}
 C(m)=\prod_{p>2}\left(1-\frac{1}{(p-1)^2}\right)^m\prod_{p}\left(1+f_m(p)\right),\nonumber
 \end{align}
where,
 \begin{align}
 f_m(2)&=\frac{1}{2}\left(\left(\frac{2}{3}\right)^m-1\right)+2^m\sum_{j\ge 2}\frac{1}{2^j}\left(\left(1-\frac{1}{2^j}\right)^m-\left(1-\frac{1}{2^{j-1}}\right)^m\right),\nonumber\\
 f_m(p)&=\frac{1}{p}\left[\left(1-\frac{1}{(p-1)^2}\right)^{-m}\left(\left(1-\frac{1}{(p-1)^2(p+1)}\right)^m+\left(\frac{p}{p-1}\right)^{m}\left(1-\frac{1}{p(p-1)}\right)^m\right)-2\right]\nonumber\\
 &\quad +\left(\frac{p}{p-1}\right)^m\left(1-\frac{1}{(p-1)^2}\right)^m\sum_{j\ge 2}\frac{1}{p^j}\left(\left(1-\frac{1}{p^j(p-1)}\right)^m-\left(1-\frac{1}{p^{j-1}(p-1)}\right)^m\right).\label{eq_C}
\end{align}
It is easy to check that $C(1)=1$. It seems difficult to simplify the expression when $m>1$.

Also, for two integers $r\le \ell $, we construct a sequence $\{d_{\ell,r}(m)\}_{m=\ell}^{\infty}$ as follows

\begin{align}
 d_{\ell,r}(m)=\sum_{k=\ell}^m\frac{k^r}{k!}\frac{(-1)^{m-k}}{(m-k)!} \label{eq_23}
\end{align}

Here we note that $d_{\ell,r}(\ell)=\frac{\ell^r}{\ell!}$; Also $d_{1,1}(1)=1$ and  $d_{1,1}(m)=0$ for $m\ge 2$.

With these notations, our next theorem is as follows

\begin{theorem}\label{Th_1}
Let $r$ and $\ell$ be two positive integers with $r\le \ell$. Also suppose $\gamma_1$ be non negative integer and $\gamma_2$ is a positive real number with $1+\gamma_1\le \gamma_2$. Now if $\C(A,B)$ be defined as in (\ref{eq_1.6}) with $A, B>x^{\frac{\ell+\gamma_1}{2}}(\log x)^{1+\ell+\gamma_2}$ and $AB>x^{\frac{3(\ell+\gamma_1)}{2}}(\log x)^{2+\ell+\gamma_2}$, then  for any positive real number $x$,
\begin{align*}
 \frac{1}{\#\C(A,B)}\sum_{{E\in \C}}\sum_{\underset{M_E(N)\ge \ell}{N\le x}}M_E(N)^r&=\sum_{m=\ell}^{\ell+\gamma_1} {C(m)d_{\ell,r}(m)} Li_m(x)+O\left(\frac{x}{(\log x)^{1+\ell+\gamma_1}}\right),
 \end{align*}
where $C(m)$ and $d_{\ell,r}(m)$ are defined in (\ref{eq_C}) and (\ref{eq_23}) respectively and $Li_m(x)=\int_{2}^x \frac{1}{(\log t)^m}\, dt$.
 \end{theorem}

We note that a theorem in \cite{Chandee_David_Koukoulopoulos_Smith:2015}, stated as Proposition \ref{prop_1} later in this paper enables one to prove Theorem \ref{Th_1} unconditionally.
Further conditionally as in Theorem \ref{thmm_1}, one has 

\begin{theorem}\label{Th_2}
 Let Conjecture \ref{conj_1} be true for some $\eta<1$. Also let $\gamma_1$ be a non negative integer and $\gamma_2>0$. Now if  $A, B>x^{\frac{\ell+\gamma_1}{2}}(\log x)^{1+\ell+\gamma_2}$ and $AB>x^{\frac{3(\ell+\gamma_1)}{2}}(\log x)^{2+\ell+\gamma_2}$, then for $r\le \ell$   
\begin{align*}\frac{1}{\#\C(A,B)}\sum_{E\in \C(A,B)}\sum_{M_E(N)\ge \ell}M_E(N)^r=\sum_{m=\ell}^{\ell+\gamma_1}{d_{\ell,r}(m)}\left(\frac{K(N)N}{\phi(N)\log N}\right)^m+O\left(\frac{N}{\phi(N)\log N}\right)^{1+\ell+\gamma_1}&\\
+O\left(\frac{1}{(\log N)^{\ell+\gamma_2}}\right),&\end{align*}
where $\C(A,B)$ is as before.
\end{theorem}

 \begin{rem}\label{rem_1}
  Recalling the fact that ${d_{\ell,r}(\ell)}=\frac{\ell^r}{\ell!}$, we note that Theorem \ref{Th_1} is somewhat similar to the prime $\ell-$tuple conjecture except for an extra $\frac{1}{\ell!}$, which comes from the permutation of a $\ell-$tuple. Also in Theorem \ref{Th_1} and Theorem \ref{Th_2}, the parameter $\gamma_1$ is introduced to express the smaller order terms with precise constants. Further in Theorem \ref{Th_2}, the implied constant in the last error term is independent of $\gamma_2$.   
 \end{rem}

In the next section we shall define the required notations and prove results that will be useful the proof of the theorems.

\section{Some results on estimation of class numbers}

Let $D$ be a negative discriminant. Using the class number formula [p. 515, \cite{Iwaniec_Kowalski:2004}], the \emph{Kronecker class number} for a discriminant $D$ can be written as 
\begin{align}
 H(D):=\sum_{\underset{D/f^2\equiv 0,1\, (\text{mod }4)}{f^2\mid D}}\frac{\sqrt{|D|}}{2\pi f}L(1, \chi_{D/f^2})\label{eq_2.1}
\end{align}
where $\chi_d$ is the Kronecker symbol $(\frac{d}{.})$ and $L(s,\chi_d):=\overset{\infty}{\underset{n=1}{\sum}}\frac{\chi_d(n)}{n^s}$.\par
Also let \begin{align}
                  D_N(p)&:=(p+1-N)^2-4p= (N+1-p)^2-4N,\label{eq_2.2}
                   \end{align}
                  Now, for $f^2\mid D_N(p)$,\text{ for } $d_{N,f}(p):=\frac{D_N(p)}{f^2}$.\\
                  Further, using Deuring's theorem\cite{Deuring:1941} we get 
   \begin{align}
   H(D_N(p))&=\sum_{\underset{|\tilde{E}(\F_p)|=N}{\tilde{E}/\F_p}}\frac{1}{\#\text{Aut}(\tilde{E})},
   \end{align}
                
where the sum is over the $F_p$-isomorphism classes of elliptic curves.\par

With these notations, we have the proposition as follows.
\begin{prop}\label{prop_1}
 Fix $R>0$, for $x\ge 1$, we have that 
 $$\frac{1}{x}\sum_{1\le N\le x}|\sum_{N^-<p<N^+}H(D_N(p))-\frac{K(N)N^2}{\phi(N)\log N}|\ll_R \frac{x}{(\log x)^R}.$$
\end{prop}
The above proposition has been proved in (Theorem 1.8, \cite{Chandee_David_Koukoulopoulos_Smith:2015}).

Note that, in our case the $p\approx N$ and $|D_N(p)|\le 4N$. With these notations, we state the following lemma

\begin{lemma}\label{lm_1.1}
Let $N$ be a positive integers and $N^-$ and $N^+$ are defined as in (\ref{eq_2}). Also  
 let $H(D_N(p))$ be defined using (\ref{eq_2.1}) and (\ref{eq_2.2}). Then  
\begin{enumerate}[(a)]
\item $$\underset{N^-<p<N^+}{\sum}{H(D_N(p))} \ll  \frac{N^2}{\phi(N)\log N}.$$
 \item Also for $k\ge 2$, $$\underset{N^-<p<N^+}{\sum}{H(D_N(p))^k} \ll  N^{\frac{k+1}{2}}(\log N)^{k-2}(\log \log N)^k.$$
\end{enumerate}
\end{lemma}

 \begin{proof}
 Part (a) essentially follows from [Theorem 1.7, \cite{Chandee_David_Koukoulopoulos_Smith:2015}]. Also see \cite{David_Smith:2013}.\par
 For part (b), We recall that
 
\begin{align}
 {H(D_N(p))}&= \sum_{\underset{\frac{D_N(p)}{f^2}\equiv 0,1 (\text{ mod }4)}{f^2\mid D_N(p)}}\frac{\sqrt{|D_N(p)|}}{2\pi f}L(1,\chi_{d_{N,f}(p)}).\nonumber
 \end{align}

Now $|D_N(p)|\le 4N$ in the above range of $p$. Also  $L(1,\chi_{d_{N,f}(p)})\ll \log N$ using convexity bound. Further using the fact that $\underset{d\mid n}{\sum}\frac{1}{d}\ll \log \log n$, we get 
 \begin{align}
 H(D_N(p))&\ll  \sum_{\underset{\frac{D_N(p)}{f^2}\equiv 0,1 (\text{ mod }4)}{f^2\mid D_N(p)}}\frac{\sqrt{N}\log N}{f} \nonumber\\
 &\ll \sqrt{N}\log N \log\log N.\label{eq_13}
 \end{align}
Then, (\ref{eq_13}) along with part (a) completes the proof.
\end{proof}
Probably a stronger bound for the second part of the previous lemma could be proved in Lemma \ref{lm_1.1}(b). But for the purpose of this paper, this result is sufficient. \par
Now, we recall the following lemma [Corollary 2F, \cite{Schimidt:1976}]:
\begin{lemma}\label{lm_1.2}
Suppose $p$ is a prime. Suppose $g(x)=a_nx^n+\cdots+a_0$ is a polynomial with integer coefficients having $0<n<p$ and $p\nmid a_n$. Then 
$$|\sum_{x=0}^{p-1}e(\frac{g(x)}{p})|\le (n-1)p^{\frac{1}{2}}.$$
\end{lemma}
\section{Curves with fixed order modulo primes}

From now on $E_{s,t}$ will denote the elliptic curve given by a Weierstrass equation of the form $y^2=x^3+sx+t$. Also if the corresponding field is of characteristic different from $2$ or $3$, then any isomorphism class of curve can be represented by one such Weierstrass equation. With these notation we state the following result   
\begin{prop}\label{prop_2}
 Let $\{p_i\}_{i=1}^\ell$ be a set of $\ell$ distinct primes in the range $[N^-,N^+]$ and $\{\tilde{E}_{s_i,t_i}/\F_{p_i}\}_{i=1}^{\ell}$ be a set of isomorphism class of elliptic curves over corresponding fields $\F_{p_i}$'s. Then for the class of rational curves $\C(A,B)$ as defined in (\ref{eq_1.6}), 
\begin{align}
 \#\{{E\in \C(A,B)}:\, E_{p_i}\cong_{p_i} \tilde{E}_{s_i,t_i} \text{ for }1\le i\le \ell \}&=\frac{4AB}{p_1\cdots p_\ell}\prod_{i=1}^\ell\left( \frac{1}{|Aut_{p_i}(E_{s_i,t_i})|}\right)+\mathcal{E}_\ell(A,B,N)
\end{align}
where \begin{align*}\mathcal{E}_\ell(A,B,N)\ll \frac{AB}{N^{2\ell}}+N^{\frac{\ell}{2}}(\log N)^2+(A\prod_{t_i=0}\sqrt{N}+B\prod_{s_i=0}\sqrt{N})N^{-\frac{\ell}{2}}\log N.\end{align*}
 \end{prop}
\begin{proof}
 
 We use a modified version of the character sum argument used by Fouvry and Murty (p. 94, \cite{Fouvry_Murty:1996}).
 First subdivide the interval $[-A,A]$ into intervals of length $p_1\cdots p_\ell$, starting from $[-A,-A+p_1p_2\dots p_\ell]$. The last one is denoted by $\A$. Similarly for $[-B,B]$, with the last one as $\B$.
Using the Chinese remainder theorem, we get 

\begin{align}
\#&\{{E\in \C(A,B)}:\, E\cong_{p_i} \tilde{E}_{s_i,t_i} \text{ for }1\le i\le \ell \}\nonumber \\ 
&=\left[\frac{2A}{p_1\cdots p_\ell}\right]\left[\frac{2B}{p_1\cdots p_\ell}\right]\prod_{i=1}^\ell\left( \frac{p_i-1}{|Aut_{p_i}(E_{s_i,t_i})}|\right)\nonumber\\
&+\left[\frac{2A}{p_1\cdots p_\ell}\right]\prod_{i=1}^\ell\frac{1}{|Aut_{p_i}(E_{s_i,t_i})|} \# \left\{(u_1,\cdots u_\ell)\in \F_{p_1}\times \cdots \times\F_{p_\ell}:   t_iu_i^6\in \B  (\text{mod } p_i) \forall 1\le i\le \ell \right\}\nonumber\\
&+\left[\frac{2B}{p_1\cdots p_\ell}\right]\prod_{i=1}^\ell\frac{1}{|Aut_{p_i}(E_{s_i,t_i})|} \# \left\{(u_1,\cdots u_\ell)\in \F_{p_1}\times \cdots \times\F_{p_\ell}:   s_iu_i^4\in \A  (\text{mod } p_i) \forall 1\le i\le \ell \right\}\nonumber\\
&+ \prod_{i=1}^\ell\frac{1}{|Aut_{p_i}(E_{s_i,t_i})|}\# \left\{(u_1,\cdots, u_\ell)\in \F_{p_1}\times \cdots \times\F_{p_\ell}:   s_iu_i^4\in \A  (\text{mod } p_i),\, t_iu_i^6\in \B  (\text{mod } p_i)   \forall 1\le i\le \ell \right\}\nonumber\\
&+O(\frac{AB}{p_1\cdots p_\ell}(\sum_{i=1}^\ell\frac{1}{p_i^9})),\label{eq_2.4}
\end{align}
where the last error term comes from the rational curves of the form $E_{s_iu_i^4p_i^4, t_iu_i^6p_i^6}$.

Now from the fourth term on the right hand side of (\ref{eq_2.4}), 
\begin{align}
&\# \left\{(u_1,\cdots u_\ell)\in \F_{p_1}\times \cdots \times\F_{p_\ell}:   s_iu_i^4\in \A  (\text{mod } p_i),\, t_iu_i^6\in \B  (\text{mod } p_i) \,  \forall 1\le i\le \ell \right\}\nonumber\\
&\quad \quad =\frac{1}{(p_1\cdots p_\ell)^2}\sum_{\underset{0\le h_i\le p_i}{(h_1,\cdots,h_\ell)}}\sum_{\underset{0\le g_i\le p_i}{(g_1,\cdots, g_\ell)}}\sum_{\underset{1\le u_i\le p_i-1}{(u_1,\cdots, u_\ell)}}\sum_{(a,b)\in \A\times \B}e\left(\sum_{i=1}^\ell\frac{h_i(s_iu_i^4-a)+g_i(t_iu_i^6-b)}{p_i}\right),\label{eq_3.1}
\end{align}

where $e(x)=e^{2\pi i x}$. \par 
When $(h_1,\cdots,h_\ell)=(0,\cdots,0)$ and $(g_1,\cdots,g_\ell)=(0,\cdots,0)$, the R.H.S of (\ref{eq_3.1}) gives a contribution equal to $|\A||\B|\overset{\ell}{\underset{i=1}{\prod}}(\frac{p_i-1}{p_i^2})$. 
Using the fact that $\A$ and $\B$ are intervals, the contributions corresponding to $(h_1,\cdots,h_\ell)\neq(0,\cdots,0)$, $(g_1,\cdots,g_\ell)\neq(0,\cdots,0)$ is bounded by 
\begin{align}\frac{1}{(p_1\cdots p_\ell)^2}\sum_{{\underset{0\le h_i\le p_i-i}{(h_1,\cdots, h_\ell)\neq (0,\cdots, 0)}}}\sum_{\underset{0\le g_i\le p_i-1}{(g_1,\cdots, g_\ell)\neq (0,\cdots, 0)}}& \left\|\frac{h_1}{p_1}+\cdots +\frac{h_\ell}{p_\ell}\right\|^{-1}\left\| \frac{g_1}{p_1}+\cdots +\frac{g_\ell}{p_\ell}\right \|^{-1}\nonumber \\
&\times \prod_{i=1}^\ell\left(\sum_{u_i=1}^{p_i-1}e\left(\frac{h_is_iu_i^4+g_it_iu_i^6}{p_i}\right)\right).\label{eq_3.2}\end{align}

If $h_ig_i$ is different from $0$ for all $i$, then $\overset{{p_i-1}}{\underset{u_i=1}{\sum}}e\left(\frac{h_is_iu_i^4+g_it_iu_i^6}{p_i}\right) \le 5\sqrt p_i$, using Lemma \ref{lm_1.2}. While if $h_{i_1}\, , h_{i_2}\,, \cdots, h_{i_r}$ are zero and other $h_i$ are non zero, then 
$$\frac{1}{(p_1\cdots p_\ell)}\sum_{\overset{\underset{0\le h_i\le p_i-i}{(h_1,\cdots, h_\ell)\neq (0,\cdots, 0)}}{h_{i_1}=h_{i_2}=\cdots=h_{i_r}=0}}\left\|\frac{h_1}{p_1}+\cdots +\frac{h_\ell}{p_\ell}\right\|^{-1}=O\left(\frac{\log \left(\frac{p_1\cdots p_\ell}{p_{i_1}\cdots p_{i_r}}\right)}{p_{i_1}\cdots p_{i_r}}\right).$$
Similar result holds for $g_i$'s.

Without loss of generality we may assume that $p_i\gg 2^{2\ell}$. In that case (\ref{eq_3.2}) is $$O(\sqrt{p_1\cdots p_\ell}\log(p_1\cdots p_\ell)^2).$$
Similarly, considering contributions corresponding to $(h_1,\cdots,h_\ell)=(0,\cdots,0)$, $(g_1,\cdots,g_\ell)\neq (0,\cdots,0)$, as well as $(h_1,\cdots,h_\ell)\neq (0,\cdots,0)$, $(g_1,\cdots,g_\ell)= (0,\cdots,0)$, (\ref{eq_3.1}) equals to

\begin{align}\label{eq_3.3}
|\A||\B|\prod_{i=1}^\ell&(\frac{p_i-1}{p_i^2})+O(\frac{|\A|}{(p_1\cdots p_\ell)}\log (p_1\cdots p_\ell)\prod_{t_i=0}(p_i)\prod_{t_i\neq 0}\sqrt(p_i)) \nonumber\\
&+O(\frac{|\B|}{(p_1\cdots p_\ell)}\log (p_1\cdots p_\ell)\prod_{s_i=0}(p_i)\prod_{s_i\neq 0}\sqrt(p_i))+O(\sqrt{p_1\cdots p_\ell}\log(p_1\cdots p_\ell)^2) 
\end{align}

Proceeding in a similar way for the second and third term in the right hand side of (\ref{eq_2.4}), we get the following

\begin{align}
 \#\{&{E\in \C(A,B)}:\, E\cong_{p_i} \tilde{E}_{s_i,t_i} \text{ for }1\le i\le \ell \} =\left[\frac{2A}{p_1\cdots p_\ell}\right]\left[\frac{2B}{p_1\cdots p_\ell}\right]\prod_{i=1}^\ell\left( \frac{p-1}{|Aut_{p_i}(E_{s_i,t_i})}\right)\nonumber\\
 &  +\left[\frac{2A}{p_1\cdots p_\ell}\right]\prod_{i=1}^\ell\frac{1}{|Aut_{p_i}(E_{s_i,t_i})|}\left[|\B|\prod_{i=1}^\ell\frac{p_i-1}{p_i}+O\left(\left(\prod_{s_i=0}p_i\right)\left(\prod_{s_i\neq0}\sqrt{p_i}\right)\log (p_1\cdots p_\ell)\right)\right]\nonumber\\
 &  +\left[\frac{2B}{p_1\cdots p_\ell}\right]\prod_{i=1}^\ell\frac{1}{|Aut_{p_i}(E_{s_i,t_i})|}\left[|\A|\prod_{i=1}^\ell\frac{p_i-1}{p_i}+O\left(\left(\prod_{t_i=0}p_i\right)\left(\prod_{t_i\neq0}\sqrt{p_i}\right)\log (p_1\cdots p_\ell)\right)\right]\nonumber\\
&  +|\A||\B|\prod_{i=1}^\ell(\frac{p_i-1}{p_i^2})+O(\frac{|\A|}{(p_1\cdots p_\ell)}\log (p_1\cdots p_\ell)\prod_{t_i=0}(p_i)\prod_{t_i\neq 0}\sqrt(p_i)) \nonumber\\
& +O(\frac{|\B|}{(p_1\cdots p_\ell)}\log (p_1\cdots p_\ell)\prod_{s_i=0}(p_i)\prod_{s_i\neq 0}\sqrt(p_i))+O(\sqrt{p_1\cdots p_\ell}\log(p_1\cdots p_\ell)^2).\nonumber
\end{align}
\begin{align}
\intertext{By combining the terms together, we get}
\#\{&{E\in \C(A,B)}:\, E\cong_{p_i} \tilde{E}_{s_i,t_i} \text{ for }1\le i\le \ell \} =\frac{4AB}{(p_1\cdots p_\ell)^2}\prod_{i=1}^\ell\left( \frac{p_i-1}{|Aut_{p_i}(E_{s_i,t_i})|}\right)\nonumber\\
&+ O(\sqrt{p_1\cdots p_\ell}\log(p_1\cdots p_\ell)^2)+O\left(\frac{A}{(p_1\cdots p_\ell)}\log (p_1\cdots p_\ell)\left(\prod_{t_i=0}p_i\right)\left(\prod_{t_i\neq 0}\sqrt p_i\right)\right)\nonumber\\
&+O\left(\frac{B}{(p_1\cdots p_\ell)}\log (p_1\cdots p_\ell)\left(\prod_{s_i=0}p_i\right)\left(\prod_{s_i\neq 0}\sqrt p_i\right)\right),\label{eqn_3.6}
\end{align}
 and this proves Proposition \ref{prop_2}.
\end{proof}

\begin{lemma}\label{lm_0.1}
Let $\C(A,B)$ be as above.
\begin{enumerate}[(a)]
 \item If $A, B>N^{\frac{\ell}{2}}(\log N)^{1+\ell+\gamma_2}$ and $AB>N^{\frac{3\ell}{2}}(\log N)^{2+\ell+\gamma_2}$, then
 \begin{align*}\frac{1}{\#\C(A,B)}\sum_{N^-<p_1\neq  \cdots \neq p_\ell<N^+}\{{E\in \C(A,B)}: &{\#E_{p_1}(\F_{p_1})=\cdots=\#E_{p_\ell}(\F_{p_\ell})=N}\}=\\
 &\left( \sum_{N^-<p< N^+}\frac{H(D_N(p))}{p} \right)^\ell+O\left(\frac{1}{(\log N)^{\ell+\gamma_2}}\right).
 \end{align*}
 \item For $r\le \ell$,
 \begin{align*}
  \frac{1}{\#\C(A,B)}\sum_{N^-<p_1,\cdots, p_r<N^+}\sum_{\underset{E_{p_1}(\F_{p_1})=\cdots=E_{p_r}(\F_{p_r})=N}{E\in \C(A,B), \, M_E(N)\ge \ell+1}}1\ll_\ell \left(\frac{H(D_N(p))}{p}\right)^{\ell+1}+\frac{1}{(\log N)^{\ell+\gamma_2}}
 \end{align*}

\end{enumerate}

   \end{lemma}

\begin{proof}

    Note that 
    \begin{align}
\#\{{E\in \C(A,B)}:\, {\#E_{p_1}(\F_{p_1})=\cdots=\#E_{p_\ell}(\F_{p_\ell})=N}\}&\nonumber\\
=\sum_{\underset{\tilde{E}_1(\F_{p_1})=N}{\tilde{E}_1/\F_{p_1}}}\cdots\sum_{\underset{\tilde{E}_\ell(\F_{p_\ell})=N}{\tilde{E}_\ell/\F_{p_\ell}}}&\#\{{E\in \C}:\, E_{p_i}\cong_{p_i} \tilde{E}_{i} \text{ for }1\le i\le \ell \}.\label{eq_2.12}
    \end{align}

    If $N>7$, then $p$ is different from $2$ and $3$. Hence every isomorphism class of curve can be represented in a minimal Weierstrass equation, say $E_{s,t}: y^2=x^3+sx+t$ with $s,t\in \F_p$. Let each of the $E_i$ are given as $E_{s_i,t_i}$.  
     so we can use Proposition \ref{prop_2} to estimate the summand in the right hand side of (\ref{eq_2.12}).

Now for a fixed prime $p_i$, the number of isomorphism class of curves $E_{s_i,t_i}$ with $s_it_i=0$ is at most $10$. Also recall that $\#\C(A,B)=4AB+O(A+B)$ and $H(D_N(p_i))=\underset{E_{s_i,t_i}/\F_{p_i}}{\sum}\frac{1}{|Aut_{p_i}(E_{s_i,t_i})|}$. 
Thus dividing (\ref{eq_2.12}) by $\C(A,B)$, the sum in the first part of the lemma equals to 
\begin{align}
 \Sigma_1&=\sum_{N^-<p_1\neq p_2\neq \cdots \neq p_\ell<N^+}\sum_{\underset{\tilde{E}_1(\F_{p_1})=N}{\tilde{E}_1/\F_{p_1}}}\cdots\sum_{\underset{\tilde{E}_\ell(\F_{p_\ell})=N}{\tilde{E}_\ell/\F_{p_\ell}}}\prod_{i=1}^\ell\frac{1}{p_i|Aut_{p_i}(E_{s_i,t_i})|}+\hat{\mathcal{E}}_\ell(A,B, N)\nonumber\\
 &= \sum_{N^-<p_1\neq p_2\neq \cdots \neq p_\ell<N^+}\left(\prod_{i=1}^\ell\frac{H(D_N(p_i))}{p_i}\right)+\hat{\mathcal{E}}_\ell(A,B, N)\label{eqn_4.2}
\end{align}
with $$\hat{\mathcal{E}}_\ell(A,B, N)\ll \left\{\frac{1}{N^{2\ell}}+\frac{\log N}{N^{\frac{\ell}{2}}}\left(\frac{1}{A}+\frac{1}{B}\right)+\frac{N^{\frac{\ell}{2}}(\log N)^2}{AB}\right\}\left(\frac{N\log \log N}{\log N}\right)^\ell$$
where the implied constant depends on $\ell$ only. 
Also since $A, B>N^{\frac{\ell}{2}}(\log N)^{1+\ell+\gamma_2}$, and $AB>N^{\frac{3\ell}{2}}(\log N)^{2+\ell+\gamma_2}$ it follows that $$\hat{\mathcal{E}}_\ell(A,B,N)\ll \frac{1}{(\log N)^{\ell+\gamma_2}}.$$

Further if we relax the condition $p_1\neq p_2\neq \cdots \neq p_\ell$ from the right hand side of (\ref{eqn_4.2}), then one gets

\begin{align}
\Sigma_1&= \sum_{\underset{N^-<p_i<N^+ \,\,  \forall i}{(p_1,p_2,\cdots, p_\ell)}}\prod_{i}\frac{H(D_N(p_i))}{p_i}+\sum_{\underset{N^-<p_i<N^+ \,\,  \forall i}{\underset{p_i=p_j \text{ for some }i\neq j}{(p_1,p_2,\cdots, p_\ell)}}}\prod_{i}\frac{H(D_N(p_i))}{p_i}+O\left(\frac{1}{(\log N)^{\ell+\gamma_2}}\right)\nonumber\\
&=\left(\sum_{N^<-p<N^+}\frac{H(D_N(p))}{p}\right)^\ell+O\left(\sum_{r=2}^\ell \left(\sum_{N^-<p<N^+}\frac{H(D_N(p))}{p}\right)^{\ell-r}\sum_{N^-<p<N^+}\frac{H(D_N(p))^r}{p^r}\right)\nonumber\\
&\quad \quad 
\quad \quad \quad \quad \quad \quad \quad \quad \quad \quad +O\left(\frac{1}{(\log N)^{\ell+\gamma_2}}\right)\nonumber\\
\end{align}     
Using Lemma \ref{lm_1.1} it is easy to see that     $$\sum_{r=2}^\ell\left(\sum_{N^-<p<N^+}\frac{H(D_N(p))^r}{p^r}\right)\left(\sum_{N^-<p<N^+}\frac{H(D_N(p))}{p}\right)^{\ell-r}\ll O(N^{-\frac{1}{2}+\epsilon})$$
for any small $\epsilon>0$.
Hence \begin{align}\Sigma_1&=\left(\sum_{N^-<p<N^+}\frac{H(D_N(p))}{p}\right)^\ell+O\left(\frac{1}{(\log N)^{\ell+\gamma_2}}\right).\label{eq_2.5}\end{align}
This proves the result part (a) of the Lemma.\par
Now, if for a curve $E$, $M_E(N)=L\ge l+1$, then $E$ is counted $L^r$ times in part (b). While the same $E$ will be counted $\frac{L!}{(L-\ell-1)!}$ times if we consider the expression
$$\frac{1}{\#\C(A,B)}\sum_{N^-<p_1\neq  \cdots \neq p_{\ell+1}<N^+}\{{E\in \C(A,B)}: {\#E_{p_1}(\F_{p_1})=\cdots=\#E_{p_{\ell+1}}(\F_{p_{\ell+1}})=N}\}$$

Using Stirling's approximation, is easy to see that $\frac{L^r(L-\ell-1)!}{L!}\ll e^{\ell}$ for $r\le \ell$. Hence part (b) follows from part (a).
\end{proof}

Using the previous lemma and modifying the proof of part (a) we shall prove an asymptotic of the left hand side of  Lemma (\ref{lm_0.1}). More precisely we state the following

\begin{prop}\label{lm_1}
Let $M_E(N)$ and $\C(A,B)$ be defined as above. If $A, B>N^{\frac{\ell+\gamma_1}{2}}(\log N)^{1+\ell+\gamma_2}$ and $AB>N^{\frac{3(\ell+\gamma_1)}{2}}(\log N)^{2+\ell+\gamma_2}$, then for any positive integer $r\le \ell$,
\begin{align*}
 \frac{1}{\#\C(A,B)}\sum_{\underset{M_E(N)\ge \ell}{E\in \C(A,B)}}M_E(N)^r=\sum_{j=\ell}^{\ell+\gamma_1}d_{\ell,r}(j)&\left(\sum_{N^-<p<N^+}\frac{H(D_N(p))}{p}\right)^{j}\\
& + O\left(\sum_p\frac{H(D_N(p))}{p}\right)^{\ell+\gamma_1+1}
 +O\left(\frac{1}{(\log N)^{\ell+\gamma_2}}\right).
\end{align*}

\end{prop}
\begin{proof}
 \begin{align}
\frac{1}{\#\C(A,B)}\sum_{\underset{M_E(N)\ge \ell}{E\in \C(A,B)}}M_E(N)^r &= \frac{1}{\#\C(A,B)}\sum_{\underset{M_E(N)\ge \ell}{E\in \C(A,B)}}\left(\sum_{\underset{E_p(\F_p)=N}{N^-<p<N^+}}1 \right)^r\nonumber\\
&= \frac{1}{\#\C(A,B)}\sum_{N^-<p_1,\cdots, p_r<N^+}\sum_{\underset{E_{p_1}(\F_{p_1})=\cdots=E_{p_r}(\F_{p_r})=N}{E\in \C(A,B), \, M_E(N)\ge \ell}}1.\nonumber
\intertext{By breaking the sum into two parts we get the following}
 \frac{1}{\#\C(A,B)}\sum_{N^-<p_1,\cdots, p_r<N^+}&\sum_{j=\ell}^{\ell+\gamma_1}\sum_{ M_E(N)= j}1+ \frac{1}{\#\C(A,B)}\sum_{N^-<p_1,\cdots, p_r<N^+}\sum_{ M_E(N)\ge \ell+\gamma_1+1}1\label{eq_24}
 \end{align}
where the range of summation is over $E\in \C(A,B)$ with ${E_{p_1}(\F_{p_1})=\cdots=E_{p_r}(\F_{p_r})=N}$.\\
 Now, by Lemma \ref{lm_0.1}(b), the last sum in the right hand side is bounded by $$\left(\sum_{N^-<p<N^+}\frac{H(D_N(p)}{p}\right)^{\ell+\gamma_1+1}+O(\frac{1}{(\log N)^{\ell+\gamma_2}})$$

Now, we claim that for $r\le\ell\le j\le \ell+\gamma_1$\begin{align}
                    \sum_{N^-<p_1\neq p_2\neq \cdots \neq p_r<N^+}\sum_{\underset{E_{p_1}(\F_{p_1})=\cdots=E_{p_r}(\F_{p_r})=N}{E\in \C(A,B), \, M_E(N)= j}}1&=\frac{1}{(j-r)!}\sum_{N^-<p_1\neq p_2\neq \cdots \neq p_j<N^+}\sum_{\underset{E_{p_1}(\F_{p_1})=\cdots=E_{p_r}(\F_{p_r})=N}{E\in \C(A,B), \, M_E(N)= j}}1\label{eq_27}
                    \intertext{In fact, any curve $E\in \C(A,B)$ with $M_E(N)=j$ is counted $\frac{j!}{(j-r)!}$ times in the left hand side summation, while on the right hand side, the same curve is counted $j!$ times.}
                    \end{align}
Note that we now consider the first term of (\ref{eq_24}), the primes in the range of summations in (\ref{eq_24}) are not distinct.                    
Then recalling the definition of $S(n,m)$, Stirling number of the second kind, which equals to the number of ways of partitioning a set of $n$ elements into $m$ nonempty sets, we get
            
\begin{align}
 \sum_{N^-<p_1,\cdots, p_r<N^+}\sum_{\underset{E(\F_{p_1})=\cdots=E(\F_{p_r})=N}{E\in \C, \, M_E(N)= j}}1
 &=\left( \sum_{m=1}^r\frac{S(r,m)}{(j-m)!}\right)\sum_{N^-<p_1\neq p_2\neq \cdots \neq p_j<N^+}\sum_{\underset{E_{p_1}(\F_{p_1})=\cdots=E_{p_r}(\F_{p_r})=N}{E\in \C(A,B), \, M_E(N)= j}}1.\label{eq_2.7}
\end{align}                    
 To simplify the constant on the right hand side, we use the fact that  $ \sum_{m=1}^r\frac{S(r,m)j!}{(j-m)!}=j^r$. See [(4.1.3), p. 60 , \cite{Roman:1984}].\par             

 With this 
\begin{align}
&\sum_{N^-<p_1\neq p_2\neq \cdots \neq p_j<N^+}\sum_{\underset{E_{p_1}(\F_{p_1})=\cdots=E_{p_r}(\F_{p_r})=N}{E\in \C(A,B), \, M_E(N)= j}}1 \nonumber\\
&=\sum_{N^-<p_1\neq p_2\neq \cdots \neq p_j<N^+}\sum_{\underset{E(\F_{p_1})=\cdots=E(\F_{p_j})=N}{E\in \C(A,B), \, M_E(N)\ge j}}1
                     -\sum_{N^-<p_1\neq p_2\neq \cdots \neq p_j<N^+}\sum_{\underset{E(\F_{p_1})=\cdots=E(\F_{p_j})=N}{E\in \C(A,B), \, M_E(N)\ge j+1}}1\label{eq_19}
\end{align}
Now we denote the left hand side of (\ref{eq_27}) by ${\#\mathcal{C}(A,B)}\times\omega(r,j)$ and the first term of the right hand side of (\ref{eq_19}) by ${\#\mathcal{C}(A,B)}\times\Omega(j,j)$. Also we call the left hand side of (\ref{eq_2.7}) by ${\#\mathcal{C}(A,B)}\times\Upsilon(r,j)$. Then in view of (\ref{eq_27}) and (\ref{eq_2.7}), we get the following set of relations
\begin{align}
\left\{ \begin{array}{ll}
\Upsilon(r,j)=\frac{j^r}{j!}\omega(j,j),\\
\Omega(t,s)=\overset{\infty}{\underset{n=s}{\sum}} \omega(t,n) \quad \text{for }t\le s,\\
\omega(t,n)=\frac{1}{(n-t)!}\omega(n,n) \quad \text{for }t\le n.
 \end{array}
\right.\label{eq_4}
\end{align}
Now by Lemma \ref{lm_0.1}(a), 
\begin{align*}
 \Omega(j,j)=\left(\sum_{N^-<p<N^+}\frac{H(D_N(p))}{p}\right)^j+O(\frac{1}{(\log N)^{j+\gamma_2}}),
\end{align*}
whenever $A,B>N^{\frac{j}{2}}(\log N)^{1+j+\gamma_2}$ and $AB>N^{\frac{3j}{2}}(\log N)^{2+j+\gamma_2}$. \par
 
 Now, we replace $\overset{\ell+\gamma_1}{\underset{j=\ell}{\sum}}\Upsilon(r,j)$ by $\overset{\ell+\gamma_1}{\underset{j=\ell}{\sum}}z_{\ell,r}(j)\Omega(j,j)+O(\Omega(\ell+\gamma_1,\ell+\gamma_1+1))$ where $\{z_{\ell,r}(j)\}$ are some constants to be determined using (\ref{eq_4}). 
Also note that $\Omega(\ell+\gamma_1,\ell+\gamma_1+1)\ll \left(\sum_p\frac{H(D_N(p))}{p}\right)^{\ell+\gamma_1}+\frac{1}{(\log N)^{\ell+\gamma_2}}$.\par
Then (\ref{eq_24}) equals to 
\begin{align*}
 \sum_{j=\ell}^{\ell+\gamma_1}z_{\ell,r}(j)\left(\sum_{N^-<p<N^+}\frac{H(D_N(p))}{p}\right)^j+O\left(\sum_{N^-<p<N^+}\frac{H(D_N(p))}{p}\right)^{\ell+\gamma_1+1}+O\left(\frac{1}{(\log N)^{\ell+\gamma_2}}\right).
\end{align*}
Only thing that remains to be shown is that $\{z_{\ell,r}(j)\}_j$ are equals to $\{d_{\ell,r}(j)\}_j$, as defined in (\ref{eq_23}). For that, we prove the following lemma.\par

\end{proof}

\begin{lemma}\label{lm_2}
Consider $\omega,\, \Omega$ as variables satisfying the identities in (\ref{eq_4}). Then the solution of the equation
$$\sum_{j=\ell}^\infty \frac{j^r}{j!}\omega(j,j)=\sum_{j=\ell}^{\infty}z_{\ell,r}(j)\Omega(j,j)$$
in $z_{\ell,r}(j)$ is given by $$z_{\ell,r}(j)=\sum_{k=\ell}^j\frac{k^r}{k!}\frac{(-1)^{j-k}}{(j-k)!}.$$
\end{lemma}
\begin{proof}
Using the second equation in (\ref{eq_4}), we have
\begin{align}
\sum_{j=\ell}^\infty \frac{j^r}{j!}\omega(j,j)&=\sum_{j=\ell}^{\infty}z_{\ell,r}(j)\Omega(j,j)\nonumber\\
 &=\sum_{j=\ell}^{\infty}z_{\ell,r}(j)\sum_{n=j}^{\infty}\omega(j,n).\nonumber
 \intertext{By changing the order of summation, the right hand side equals to}
 &=\sum_{n=\ell}^\infty\sum_{\ell\le j\le n}z_{\ell,r}(j)\omega(j,n)=\sum_{j=\ell}^\infty\sum_{\ell\le n\le j}z_{\ell,r}(n)\omega(n,j)\nonumber
 \intertext{But by the last relation in (\ref{eq_4}), this can be written as }
 &\sum_{j=\ell}^\infty \left(\sum_{\ell\le n\le j}\frac{z_{\ell,r}(n)}{(j-n)!}\right)\omega(j,j)\nonumber
\end{align}
Thus, comparing the coefficients of $\omega(j,j)$ from both sides, we get
\begin{align}
 \sum_{\ell\le n\le j}\frac{z_{\ell,r}(n)}{(j-n)!}=\frac{j^r}{j!} \quad \quad\text{ for }j\ge \ell.\label{eq_6}
\end{align}
Since we are only interested in the values of $z_{\ell,r}(n)$ for $\ell \le n\le \ell+\gamma_1$, we consider the folowing matrix equation 
\begin{align*}
AZ=J, 
\end{align*}
where $A$ is the $(\ell+\gamma_1+1)\times \ (\ell+\gamma_1+1)$ matrix $\left(a_{mn}\right)_{m,n}$, where 
\begin{align*}
 a_{mn}&=\left\{\begin{array}{ll}
          0,\quad \text{ if }m<n,\\
          \frac{1}{(m-n)!} \quad \text{ if } m\ge n; 
         \end{array}
\right.
\end{align*}
Also  $Z$ and $J$ are the column matrices $$\begin{bmatrix} z_{\ell,r}(\ell) & z_{\ell,r}(\ell+1) &\cdots & z_{\ell,m}(\ell+\gamma_1)\end{bmatrix}^{\mathrm{T}}$$ and 
$$\begin{bmatrix} \frac{\ell^r}{\ell!} & \frac{(\ell+1)^r}{(\ell+1)!} &\cdots & \frac{(\ell+\gamma_1)^r}{(\ell+\gamma_1)!}\end{bmatrix}^{\mathrm{T}}$$ respectively.\par 

Now it is not difficult to check that $A$ s as invertible matrix with inverse $B=\left(b_{mn}\right)$, where  
\begin{align*}
 b_{mn}=(-1)^{m-n}a_{mn}. 
\end{align*}
Finally, using $Z=A^{-1}J=BJ$, we get the desired value of $z_{\ell,r}(j)$'s. This completes the proof of the lemma.

\end{proof}

\section{Proof of Theorem \ref{th_1} and Theorem \ref{Th_2}}
Putting $\ell=1$, $r=1$ and $\gamma_1=0$, $\gamma_2=\gamma$, from Proposition \ref{lm_1} we get,
\begin{align}
\frac{1}{\#\C(A,B)}\sum_{E\in \C(A,B)}M_E(N)=\sum_{N^-<p<N^+}\frac{H(D_N(p))}{p}+O\left(\left(\sum_{N^-<p<N^+}\frac{H(D_N(p))}{p}\right)^2\right)\nonumber\\
+O(\frac{1}{(\log N)^{1+\gamma}})& \label{eq_7}
\end{align}
for appropriate $A$, $B$. Then, using (\ref{eq_7}), we replace $\underset{N^-<p<N^+}{\sum}\frac{H(D_N(p))}{p}$ in \emph{Proposition \ref{lm_1}} by $\frac{1}{\#\C(A,B)}\sum_{E\in \C(A,B)}M_E(N)$. We also recall that $d_{\ell,r}(\ell)=\frac{\ell^r}{\ell!}$. 
Now take $\gamma_1=0$, $r=1$ and consider the sum $\frac{1}{\#\C(A,B)}\underset{\underset{M_E(N)=\ell}{E\in \C(A,B)}}{\sum}M_E(N)=\frac{1}{\#\C(A,B)}\underset{\underset{M_E(N)=\ell}{E\in \C(A,B)}}{\sum} \ell$. 
Then dividing the last equation by $\ell$, \emph{Theorem \ref{th_1}} follows immidiately from the above discussion.\par \medskip Again, (\ref{eq_7}) together with Proposition \ref{lm_1} and Theorem \ref{thmm_1} completes the proof of  Theorem \ref{Th_2}.

\section{Proof of Theorem \ref{Th_1}}
First of all note that
\begin{align*}
\sum_{N^-<p<N^+}\frac{H(D_N(p))}{p}
 &=\frac{1}{N}\sum_{N^-<p<N^+}H(D_N(p))\left(1+O\left(\frac{1}{\sqrt{N}}\right)\right)\\
 &=\frac{1}{N}\sum_{N^-<p<N^+}H(D_N(p))+\frac{1}{N^{\frac{3}{2}}}\sum_{N^-<p<N^+}|H(D_N(p))|
\end{align*}
Now from Lemma \ref{lm_1.1}(a), we get 
 \begin{align*}
\sum_{N^-<p<N^+}\frac{H(D_N(p))}{p}&=\frac{1}{N}\sum_{N^-<p<N^+}H(D_N(p))+O\left(\frac{\log \log N}{\sqrt{N}\log N}\right)
\intertext{Also }
\left(\sum_{N^-<p<N^+}\frac{H(D_N(p))}{p}\right)^j&=\frac{1}{N^j}\left(\sum_{N^-<p<N^+}H(D_N(p))\right)^j+O\left(\frac{1}{\sqrt{N}}\right)
\end{align*}
Then 
\begin{align*}
\sum_{N\le x}\left(\sum_{N^-<p<N^+}\frac{H(D_N(p))}{p}\right)^j
&=\sum_{N\le x}\frac{1}{N^j}\left(\sum_{N^-<p<N^+}H(D_N(p))\right)^j+O(\sqrt{x})\\
&=\sum_{N\le x}\left(\frac{K(N)N}{\phi(N)\log N}\right)^j+\tilde{\mathcal{E}}_1
\end{align*}

To bound the error $\tilde{\mathcal{E}}_1$, note that
\begin{align*}
\tilde{\mathcal{E}}_1\ll& \sum_{N\le x}\frac{1}{N^j} \left|\left(\sum_{N^-<p<N^+}H(D_N(p))\right)^j-\left(\frac{K(N)N^2}{\phi(N)\log N}\right)^j\right|+O(\sqrt{x})
\intertext{ Using Lemma \ref{lm_1.1}(a), the right hand side is bounded by}
 & \sum_{N\le x}\frac{1}{N^j}\left|\sum_{N^-<p<N^+}H(D_N(p))-\frac{K(N)N^2}{\phi(N)\log N}\right|\left(\frac{N^2}{\phi(N)\log N}\right)^{j-1}+O(\sqrt{x})\\
\ll & \frac{1}{x}\sum_{N\le x}\left|\sum_{N^-<p<N^+}H(D_N(p))-\frac{K(N)N^2}{\phi(N)\log N}\right|+\sqrt{x}
\end{align*}

Using Proposition \ref{prop_1} with $R=1+\ell+\gamma_1$, the last summation is 
 $$ \ll_{\ell,\gamma_1}  \frac{x}{(\log x)^{1+\ell+\gamma_1}}+\sqrt{x}.$$ 

 Only thing that remains is to estimate the main term, i.e. $$\sum_{N\le x}\left(\frac{K(N)N}{\phi(N)\log N}\right)^j$$
for every $\ell\le j\le \ell+\gamma_1$.
To do this, we write 
$$\left(\frac{K(N)N}{\phi(N)}\right)^j=\varTheta F(N-1)G(N)$$
where \begin{align*}
       \varTheta&=\prod_{p>2}\left(1-\frac{1}{(p-1)^2}\right)^j\\
       F(N)&=\prod_{\underset{p>2}{p\mid N}}\left(1-\frac{1}{(p-1)^2}\right)^{-j}\prod_{p\mid N}\left( 1-\frac{1}{(p-1)^2(p+1)} \right)^j\\
       G(N)&=\left(\frac{N}{\phi(N)}\right)^j\prod_{\underset{p>2}{p\mid N}}\left( 1-\frac{1}{(p-1)^2}\right)^{-j}\prod_{p\mid N}\left(1-\frac{1}{p^{\nu_p(N)}(p-1)}\right)^j
      \end{align*}
Note that both $F$ and $G$ are multiplicative functions.
We use Theorem 1 of \cite{Balasubramanian_Giri:2015} with $A(n)=B(n)=1$, and hence $M(x)=x$. Also if we set  
  \begin{align}f(m)&=\underset{d\mid m}{\sum}\mu(d)F(m/d)\label{eq_17}\\ \intertext{and}
  g(m)&=\underset{d\mid m}{\sum}\mu(d)G(m/d),\label{eq_18}\end{align}
  then $f, g$ are multiplicative functions. So it is enough to compute the values on prime powers. It is straight forward to check that 
 \begin{align*}
 f(p^t)&=\left \{ \begin{array}{ll}
                         1, \quad \quad\text{if }t=0\\
                         \left(1-\frac{1}{(p-1)^2}\right)^{-j}\left( 1-\frac{1}{(p-1)^2(p+1)} \right)^j-1, \quad \text{if }t=1\\
                         0, \quad \quad \text{else,}\\
                        \end{array}
\right. \\ \intertext{and}
g(p^t)&=\left \{\begin{array}{ll}
                 1, \quad \text{if }t=0\\
                 (\frac{p}{p-1})^j(1-\frac{1}{(p-1)^2})^{-j}(1-\frac{1}{p(p-1)})^j-1, \quad \text{if }t=1\\
                 (\frac{p}{p-1})^j(1-\frac{1}{(p-1)^2})^{-j}[(1-\frac{1}{p^{t}(p-1)})^j-(1-\frac{1}{p^{t-1}(p-1)})^j],\quad \text{ if }t\ge 2,
                \end{array}
\right.
 \end{align*}
  for an odd prime $p$.\par
  Also
  \begin{align*}f(2^t)&=\left \{ \begin{array}{ll}
                         (2/3)^j-1, \quad \quad\text{if }t=1\\
                         0, \quad \quad \text{if }t\ge 2,\\
                        \end{array}
\right. \\ \intertext{and}
g(2^t)&=\left \{ \begin{array}{ll}
                 0, \quad \text{for }t= 1\\
                         2^{j}[(1-\frac{1}{2^t})^j-(1-\frac{1}{2^{t-1}})^j], \quad \quad \text{if }t\ge 2.\\
                        \end{array}
\right. \end{align*}

Then from (Theorem 1, \cite{Balasubramanian_Giri:2015}), we know  
$$\frac{1}{x}\sum_{N\le x}\left(\frac{K(N)N}{\phi(N)}\right)^j= \varTheta \sum_{N\le x}F(N-1)G(N)=\varTheta \prod_{p}\left( 1+\sum_{t\ge 1}\frac{f(p^t)+g(p^t)}{p^t}\right)+O\left(\frac{\log x}{x}\right).$$
But the constant in the main term is nothing but the $C(j)$, which has been defined in (\ref{eq_C}). 
 Using partial summation we get 
 $$\sum_{N\le x}\left(\frac{K(N)N}{\phi(N)\log N}\right)^j=C(j)\int_{2}^x\frac{1}{(\log t)^j}\, dt+O\left(\frac{x}{(\log x)^{R_1}}\right)$$ for any $R_1>0$. By choosing $R_1=1+\ell+\gamma_1$ we completes the proof of Theorem \ref{Th_1}.\par

\section*{Acknowledgement}
The second author would like to thank Chantal David and Dimitris Koukoulopoulos for some fruitful discussions and their  suggestions regarding the presentation of this paper.

\bibliographystyle{amsalpha}

\end{document}